\documentclass{amsart}
\usepackage{amsfonts}
\usepackage{mathrsfs}  
 \usepackage{amsmath}    
 \usepackage{amsthm}     
 \usepackage{amscd}      
 \usepackage{amssymb}    
 \usepackage{eucal}      
 \usepackage{latexsym}   
 \usepackage{graphicx}   
 \usepackage{verbatim}   

\usepackage[all]{xy}   
\newcommand{\nN}{\mathbb{N}}                     
\newcommand{\nZ}{\mathbb{Z}}                     
\newcommand{\nR}{\mathbb{R}}                     
\newcommand{\nQ}{\mathbb{Q}}                     

\newcommand{\nP}{\mathbb{P}}                     

\newcommand{\sF}{\mathscr{F}}
\newcommand{\sO}{\mathscr{O}}                    
\newcommand{\sI}{\mathscr{I}}                    
\newcommand{\mJ}{\mathcal{J}}                    

\DeclareMathOperator{\Bl}{Bl}                    
\DeclareMathOperator{\mld}{mld}                  
\DeclareMathOperator{\Supp}{Supp}                
\DeclareMathOperator{\reg}{reg}                  
\DeclareMathOperator{\codim}{codim}              
\newcommand{\ord}{\mbox{ord}}                    
\DeclareMathOperator{\val}{Val}                  
\DeclareMathOperator{\Val}{Val}                  

\DeclareMathOperator{\Nlci}{Nlci}                  

\newcounter{IntroThm}                           
\newtheorem{introthm}[IntroThm]{Theorem}        
\newtheorem{introcor}[IntroThm]{Corollary}

\newtheorem{theorem}{Theorem}[section]

\newtheorem{lemma}[theorem]{Lemma}
\newtheorem{corollary}[theorem]{Corollary}
\newtheorem{conjecture}[theorem]{Conjecture}
\newtheorem*{Fujita}{Fujita's Vanishing Theorem}
\newtheorem*{Nadel}{Nadel's Vanishing Theorem}

\theoremstyle{definition}
\newtheorem{definition}[theorem]{Definition}
\newtheorem{remark}[theorem]{Remark}

\newtheorem{example}[theorem]{Example}

\numberwithin{equation}{section}

\begin{document}
\title[A Vanishing Theorem and Asymptotic Regularity]{A Vanishing Theorem and Asymptotic Regularity of Powers of Ideal Sheaves}

\author{Wenbo Niu}
\address{Department of Mathematics, Statistics, and Computer
Science, University of Illinois at Chicago, 851 South Morgan Street,
Chicago, IL 60607-7045, USA}

\email{wniu2@uic.edu}



\date{}
\maketitle
\begin{abstract} Let $\sI$ be an ideal sheaf on $\nP^n$. In the first
part of this paper, we bound the asymptotic regularity of powers of
$\sI$ as $sp\leq \reg \sI^p\leq sp+e$, where $e$ is a constant and
$s$ is the $s$-invariant of $\sI$. We also give the same upper bound
for the asymptotic regularity of symbolic powers of $\sI$ under some
conditions. In the second part, by using multiplier ideal sheaves,
we give a vanishing theorem of powers of $\sI$ when it defines a
local complete intersection subvariety with log canonical
singularities.
\end{abstract}
\section{introduction}
\noindent Throughout this paper, we work over an algebraically
closed field with characteristic zero. Given an ideal sheaf $\sI$ on
the projective space $\nP^n$, one invariant to measure its
complexity is the Castelnuovo-Mumford regularity (or simply
regularity), denoted by $\reg \sI$. It was introduced by Mumford
in \cite[Chapter 14]{Mumford:LecCurOnSurf} and defined as the
minimal number $m$ such that $H^i(\nP^n,\sI(m-i))=0$ for all $i>0$.

The positivity of $\sI$ is measured by the $s$-invariant which is
the reciprocal of the Seshadri constant of $\sI$ with respect to the
hyperplane divisor and denoted by $s=s(\sI)$. The result of
Cutkosky, Ein and Lazarsfeld \cite{Ein:PosiComlIdeal} has shown that
asymptotically $\reg \sI^p$ is a linear-like function with the slope
$s$, that is
$$\lim_{p\rightarrow \infty}\frac{\reg\sI^p}{p}=s.$$

An interesting question is, for $p$ sufficiently large, if $\reg
\sI^p$ will be actually a linear function, namely $\reg \sI^p=sp+e$
for some constant number $e$. The same question has been answered in
commutative algebra for homogeneous ideals. Precisely, given an
arbitrary homogenous ideal $I$ of the polynomial ring, for $p$
sufficiently large there exist constants $d$ and $e$ such that $\reg
I^p=dp+e$ \cite[Theorem 3.1]{Cutkosky:AsymReg}, \cite[Theorem
5]{Kodiyalam:AymReg}. However, for the ideal sheaf $\sI$ this result
is not true in general. Several examples (e.g.
\cite{Cutkosky:AsymReg}, \cite{Ein:PosiComlIdeal}) have shown that
$\reg \sI^p$ is far from being a linear function even when $p$ is
sufficiently large. This is mainly due to the fact that $s$ could be irrational.
Thus the best one can hope for is that the difference between $\reg
\sI^p$ and $sp$ can be bounded by constants
\cite{Cutkosky:AsyRegPts}.

In the first part of this paper, we show that the asymptotic
regularity of $\sI$ can indeed be bounded by linear functions of the
slope $s$.

\begin{introthm} Let $\sI$ be an ideal sheaf on $\nP^n$ and let
$s=s(\sI)$ be the $s$-invariant. Then there exists a constant $e$
such that for all $p\geq 1$, one has
$$sp\leq \reg \sI^p\leq sp +e.$$
\end{introthm}

The main idea to prove this result is to use Fujita's vanishing
theorem, which is a variant of Serre's vanishing theorem, on the
blowing-up of $\nP^n$ along the ideal $\sI$. This idea has been used
successfully in \cite{Ein:PosiComlIdeal} and still has its value in
the study of asymptotic regularity.

Let $d$ be a positive integer such that $\sI(d)$ is generated by its
global sections, then it is easy to see $s\leq d$. Thus our method
also gives a geometric proof of the result $\reg \sI^p\leq dp +e$
which was found by a commutative algebraic method in
\cite{Swanson:PowOfIdelas}, \cite{Cutkosky:AsymReg} and
\cite{Kodiyalam:AymReg}.

By comparing its ordinary and symbolic powers of an ideal sheaf,
we are able to give an estimation on the asymptotic regularity of symbolic
powers. As a flavor of this estimation, we give the following result
of lower dimensional varieties. For more general results, see
Theorem \ref{thm45} in section 2.

\begin{introthm} Let $\sI$ be an ideal sheaf defining a reduced
subscheme $C$ of $\nP^n$ of dimension $\leq 1$ and let $s=s(\sI)$ be
the $s$-invariant. Then there exists a constant $e$ such that for
all $p\geq 1$, one has
$$\reg \sI^{(p)}\leq sp +e.$$
\end{introthm}

In the second part of this paper, we turn to prove a vanishing
theorem of powers of ideal sheaves. A surprising result due to
Bertram, Ein and Lazarsfeld \cite[Theorem 1]{BEL} says that if $\sI$
defines a nonsingular subvariety $X$ of codimension $e$ in $\nP^n$,
cut out scheme-theoretically by hypersurfaces of degrees $d_1\geq
d_2\geq \cdots \geq d_t$, then $H^i(\nP^n,\sI^p(k))=0$ for $i>0$ and
$k\geq pd_1+d_2+\cdots+d_e-n$ and consequently one has a linear
bound $\reg \sI^p\leq pd_1+d_2+\cdots+d_e-e+1$.

This result has led to much research on finding linear bounds for
the regularity of a homogeneous ideal in terms of its generating
degrees. Such bounds turn out to be very sensitive to the
singularities of $X$. For $p=1$, the same bound has
been established for a local complete intersection subvariety with
rational singularities by Chardin and Ulrich \cite{CU:Reg}. The recent work of Ein and deFernex \cite{Ein:VanishLCPairs} generalizes this bound to the case where the pair $(\nP^n,eX)$ is log canonical. Although in the same paper, Ein and deFernex has established a corresponding vanishing theorem of the ideal sheaf, they still left the question to establish a vanishing theorem of powers of the ideal
sheaf for all $p\geq 1$.

Inspired by the work of Ein and deFernex, we establish a vanishing
theorem of powers of an ideal sheaf, which also generalizes
\cite{BEL}.

\begin{introthm} Let $X$ be a nonsingular variety and $V\subset X$ be
a local complete intersection subvariety with log canonical
singularities. Suppose that $V$ is scheme-theoretically given by
$$V=H_1\cap\cdots\cap H_t,$$
for some $H_i\in|L^{\otimes d_i}|$, where $L$ is a globally
generated line bundle on $X$ and $d_1\geq\cdots\geq d_t$. Set
$e=\codim_X V$, then we have
$$H^i(X,\omega_X\otimes L^{\otimes k}\otimes A\otimes \sI^{p}_V)=0,\mbox{ for $i>0,\ k\geq pd_1+d_2+\cdots+d_e$},$$
where $p\geq1$ and $A$ is a nef and big line bundle on $X$.
\end{introthm}

To prove this result, we mainly follow the idea in
\cite{Ein:VanishLCPairs} to construct a formal sum
$Z=(1-\delta)B+\delta e V+(p-1)V$, for $0<\delta\ll 1$ and $p\geq1$,
where $B$ is the base scheme of some linear series. Then at a
neighborhood of $V$, the associated multiplier ideal sheaf
$\mJ(X,Z)$ is equal to $\sI^p$. This gives us a chance to apply
Nadel's vanishing theorem to the multiplier ideal sheaf $\mJ(X,Z)$,
from which we are able to deduce the vanishing theorem of $\sI^p$.

Having the above vanishing theorem in hand and applying it to a subvariety of $\nP^n$, we obtain a linear bound for the regularity of
powers.

\begin{introcor}\label{intro:cor} Let $V\subset \nP^n$ be a subvariety such that $V$
is a local complete intersection with log canonical singularities.
Assume that $V$ is cut out scheme-theoretically by hypersurfaces of
degrees $d_1\geq \cdots\geq d_t$ and set $e=\codim V$. Then
$$H^i(\nP^n,\sI^p_V(k))=0,\mbox{ for $i>0,\ k\geq pd_1+d_2+\cdots+d_e-n$}.$$
In particular, one has
$$\reg \sI^p\leq pd_1+d_2+\cdots+d_e-e+1.$$
\end{introcor}
Thus our result provides a new reasonable geometric condition so
that a linear bound for the regularity can be established.

\vspace{0.5cm}

\noindent {\em Acknowledgement. }Special thanks are due to the
author's advisor Lawrence Ein who introduced the author to
this subject and has offered a lot of help and suggestions. The author also would like to thank Brian Harbourne for his
suggestions on comparing powers and symbolic powers of ideals, Christian Schnell for his suggestions which improves the lower bound in the first theorem, and the referees for their patient reading and kind suggestions which improve the quality of the paper.

\section{Asymptotic regularity of ideal sheaves}
\noindent In this section, we bound the asymptotic regularity of
powers of an ideal sheaf by linear functions, whose slope are the
$s$-invariant of the ideal sheaf. We start by recalling the
definitions of regularity and $s$-invariant. A good reference for these topics is Section 1.8 and Section 5.4 of the book of  Lazarsfeld \cite{Lazarsfeld:PosAG1}.

\begin{definition} A coherent sheaf $\sF$ on $\nP^n$ is $m$-regular
if
$$H^i(\nP^n,\sF(m-i))=0,\quad\quad \mbox{ for\  }\  i>0.$$
The {\em regularity} $\reg(\sF)$ of $\sF$ is the least integer $m$
for which $\sF$ is $m$-regular.
\end{definition}

Given an ideal sheaf $\sI$ on $\nP^n$, consider the blowing-up
$$\mu:W=\Bl_{\sI}\nP^n\rightarrow \nP^n$$
of $\nP^n$ along the ideal $\sI$ with the exceptional Cartier
divisor $E$ on $W$, such that $\sI\cdot\sO_W=\sO_W(-E)$. Let $H$ be
the hyperplane divisor of $\nP^n$. Note that for $m$ sufficiently
large $m\mu^*H-E$ is ample on $W$, since $\sO_W(-E)$ is $\mu$-ample.

\begin{definition} We define the {\em $s$-invariant} of
$\sI$ to be the positive real number
$$s(\sI)=\min\{\ s\ |\ s\mu^*H-E \ \mbox{ is nef }\}.$$
Here $ s\mu^*H-E$ is considered as an $\nR$-divisor on $W$.
\end{definition}

\begin{remark} In fact, fix any
ample divisor $H$ on a nonsingular projective variety $X$, we can
define the $s$-invariant $s_H(\sI)$ and the regularity $\reg_H(\sF)$
with respect to $H$ for any ideal sheaf $\sI$ of $\sO_X$ and any
coherent sheaf $\sF$ on $X$. For example, this generalization has been considered in \cite{Ein:PosiComlIdeal}.
However, for simplicity, in this paper, we stick to the case of
$X=\nP^n$ and fix $H$ as the hyperplane divisor of $\nP^n$ and just
write $s_H(\sI)$ as $s(\sI)$. It is not hard to deal with the
general case by applying our method here directly.
\end{remark}

The following vanishing theorem of Fujita will be used in the proof
of our main theorem. It is a generalization of Serre's vanishing
theorem. A detailed proof can be found in \cite[Theorem 1.4.35]{Lazarsfeld:PosAG1}.

\begin{Fujita} Let $V$ be a projective variety. Fix $A$ an ample
divisor and $\sF$ a coherent sheaf. There is a number
$m_0=m_0(A,\sF)$ such that for any nef divisor $B$,
$$H^i(V,\sF(mA+B))=0 \mbox{ \ \ \ for }i>0,m\geq m_0.$$
\end{Fujita}

Notice that the crucial point in the above theorem is that the number $m_0$ only depends on the ample divisor $A$ and the coherent sheaf $\sF$ but not on the nef divisor $B$.

\begin{theorem}\label{b2} Let $\sI$ be an ideal sheaf on $\nP^n$ and let
$s=s(\sI)$ be the $s$-invariant. Then there exists a constant $e$
such that for all $p\geq 1$, one has
$$sp\leq \reg \sI^p\leq  sp +e.$$
\end{theorem}
\begin{proof} We first prove the upper bound of $\reg \sI^p$. For this, it
suffices to show that there exists a constant $e$ such that for all
$p\geq 1$, we have
$$\reg \sI^p\leq \lceil sp\rceil +e.$$
where $\lceil sp\rceil$ means the least integer greater than $sp$.

Consider the blowing-up $\mu:W=\Bl_{\sI}(\nP^n)\rightarrow \nP^n$ of
$\nP^n$ along $\sI$, with the exceptional divisor $E$. Let
$H=\sO_{\nP^n}(1)$ be the ample hyperplane divisor.

We can choose a rational number $\epsilon$ such that
$$(\lceil s\rceil +\epsilon )\mu^*H-E \mbox{\ \ is ample}.$$
By considering the sheaf $\sF=\sO_W$ in Fujita's Vanishing Theorem, we see that there is a large
integer $n_0$ such that $n_0\epsilon$ is an integer and the ample
divisor
$$A=n_0(\lceil s\rceil +\epsilon )\mu^*H-n_0E$$
satisfies Fujita's Vanishing Theorem for any nef line bundle. Note
that we can write
$$n_0(\lceil s\rceil +\epsilon )=\lceil n_0s\rceil+e_0$$
for some non-negative constant $e_0$ and therefore $A=(\lceil n_0s\rceil+e_0)\mu^*H-n_0E$. We fix such $n_0$, $e_0$ and the ample divisor $A$ in the sequel.

Now for an integer $p$ large enough, say larger than $n_0$, we consider a divisor $B_p$ defined as $$B_p=\lceil(p-n_0)s\rceil\mu^* H-(p-n_0)E.$$ Then $B_p$ is nef because of the definition of $s$ and the inequality
$$\frac{\lceil(p-n_0)s\rceil}{p-n_0}\geq \frac{(p-n_0)s}{p-n_0}=s.$$
Now we add this nef divisor $B_p$ to the ample divisor $A$ constructed above to get the divisor  $$A+B_p=(\lceil n_0s\rceil+\lceil(p-n_0)s\rceil+e_0)\mu^*H-pE.$$

Notice that the divisor $A+B_p$
has no higher cohomology because of the choice of $A$ and Fujita's Vanishing Theorem. It is an easy fact that for any positive real numbers $a$ and $b$, $\lceil
a\rceil+\lceil b\rceil=\lceil a+b\rceil+c$ where $c=0$ or $1$. Thus we can write $\lceil n_0s\rceil+\lceil(p-n_0)s\rceil=\lceil n_0s+(p-n_0)s\rceil+c=\lceil sp\rceil+c$ where $c=0$ or $1$ and then the divisor $A+B_p=(\lceil sp\rceil + e_0+c)\mu^*H-pE$. Finally by adding an additional $\mu^*H$ to $A+B_p$ when $c=0$ we obtain a divisor $$R_p=A+B_p+(1-c)\mu^*H=(\lceil sp\rceil + e_0+1)\mu^*H-pE$$ (this possible extra $\mu^*H$ is just for canceling the awkward number $c$). Since $\mu^*H$ is nef the divisor $R_p$ does not have any higher cohomology by the choice of $A$ and Fujita's Vanishing Theorem again.
That means we get
$$H^i(W, \sO_W((\lceil sp\rceil + e_0+1)\mu^*H-pE))=0\quad\quad\mbox{for } i>0 \mbox{\quad and }p\gg0.$$
Thus by \cite[Lemma 3.3]{Ein:PosiComlIdeal}, there is a number $p_0$
such that for $p>p_0$, we have
$$H^i(\nP^n,\sI^p(\lceil ps\rceil + e_0+1))=0\mbox{\ \ \ for }i>0.$$
Therefore $\sI^p$ is $(\lceil ps\rceil+e_0+1+n)$-regular. Taking
into account the finitely many cases where $p\leq p_0$, we can have
a constant $e$ such that $\reg \sI^p\leq \lceil ps\rceil+e$ for all
$p\geq 1$.

Next, we prove the lower bound of $\reg \sI^p$. For $p\geq 1$, suppose $r_p=\reg \sI^p$. Then $\sI^p(r_p)$ is generated by its global sections. Thus the invertible sheaf $\sO_W(r_p\mu^*H-pE)$ is also generated by
its global sections and in particular is nef. Hence by the definition of $s$, we get $r_p/p\geq s$, that is $r_p\geq sp$. So we get the lower bound $\reg \sI^p\geq sp$.

Combining arguments together we can find a constant $e$ such that $sp\leq \reg \sI^p\leq \lceil sp\rceil+e$ from which the theorem follows.
\end{proof}

\begin{remark} (1) In the first draft of this paper, the lower bound of the theorem is $sp-3$. However, Christian Schnell pointed out that we actually can improve it by just applying \cite[Lemma 1.4]{Ein:PosiComlIdeal}. Since the argument is very short, we included it in the proof for the convenience of the reader.

(2) Note that in the proof of the theorem, we only need to
use Fujita's Vanishing Theorem on the blowing-up and the property of
the $s$-invariant. The same idea has appeared in
\cite{Ein:PosiComlIdeal}. We hope that this idea will still be
useful in the study of the asymptotic regularity.

(3) Let $d$ be an integer such that $\sI(d)$ is generated by its
global sections. Then it is easy to see $s\leq d$, and we obtain
immediately from the theorem that $\reg\sI^p\leq dp+e$ for some
constant $e$. Thus our approach gives a geometric proof of the
result of linear bounds for the asymptotic regularity of an ideal
sheaf obtained in \cite{Swanson:PowOfIdelas}, \cite{Cutkosky:AsymReg} and \cite{Kodiyalam:AymReg} by means of commutative algebra.

(4) Following notation in \cite{Cutkosky:AsyRegPts}, we define a
function $\sigma_{\sI}:\nN\rightarrow \nZ$ by
$$\reg \sI^p=\lfloor sp\rfloor +\sigma_{\sI}(p),$$
for the ideal sheaf $\sI$. From the proof of the above theorem, we
can find a constant $e$ such that
$$0\leq \sigma_{\sI}(p)\leq e.$$
This answers the question of determining whether $\sigma_{\sI}(p)$
is bounded, which is  proposed in \cite{Cutkosky:AsyRegPts}. Furthermore we see that the function $\sigma_{\sI}(p)$ is always positive, which has been showed in \cite{Cutkosky:AsymReg} and \cite{Cutkosky:AsyRegPts} for some specific examples.

(5) Still keep notation in the proof of the theorem. Let
$f:W^+\rightarrow W$ be the normalization of $W$ and let
$\nu:W^+\rightarrow \nP^n$ be the composition of $\mu\circ f $ and
denote by $F$ the exceptional divisor on $W^+$ such that $\sI\cdot
\sO_{W^+}=\sO_{W^+}(-F)$. The {\em integral closure}
$\overline{\sI}$ of $\sI$ is defined by the ideal
$\nu_*\sO_{X^+}(-F)$. For any $p\geq 1$, the integral closure
$\overline{\sI^p}$ of $\sI^p$ is then equal to
$\nu_*\sO_{X^+}(-pF)$. Note that since $f$ is finite and $\sO_W(-E)$
is $\mu$-ample, $\sO_{W^+}(-F)=f^*\sO_W(-E)$ is $\nu$-ample and for
any real number $\epsilon$, $\epsilon\nu^*H-F$ is ample on $W^+$ if
and only if $\epsilon\mu^*H-E$ is ample on $W$. This implies that
$s(\sI)=s(\overline{\sI})$. Thus the proof the the theorem works for
the integral closure $\overline{\sI^p}$, and therefore there exists a constant
$e$ such that
$$sp\leq \reg \overline{\sI^p}\leq  sp +e.$$
In particular, we have the limit
$$\lim_{p\rightarrow\infty}\frac{\reg\overline{\sI^p} }{p}=s.$$
\end{remark}

\vspace{0.3cm}

In the rest of this section, as an application of the theorem above, we turn to bounding the asymptotic regularity of symbolic powers of an ideal sheaf. Assume in the
sequel that $\sI$ is an ideal sheaf on a nonsingular variety $X$
(not necessarily projective) and it defines a reduced subscheme
$Z$ of $X$. We start with recalling the definition of symbolic
powers of $\sI$.

\begin{definition}\label{b1} The $p$-th symbolic power of $\sI$ is
the ideal sheaf consisting of germs of functions that have
multiplicity $\geq p$ at each generic point of $Z$, i.e.,
$$\sI^{(p)}=\{f\in \sO_X\ |\ f\in m^p_{\eta}\quad\mbox{ for each generic point $\eta$ of $Z$}\},$$
where $m_{\eta}$ means the maximal ideal of the local ring
$\sO_{X,\eta}$.
\end{definition}

It is easy to see that if $Z$ has dimension zero, i.e., it consists
of distinct points on $X$, then $\sI^p=\sI^{(p)}$ for all $p\geq 1$.
But if $Z$ has positive dimension, then for any $p\geq 1$, we can
only  have $\sI^p\subseteq \sI^{(p)}$ and the inclusion is strict in
general. However, a surprising result due to Ein, Lazarsfeld and Smith \cite{Ein:UniBdSymPw} says that if $e\geq \codim_XZ$, then
$\sI^{(ep)}\subseteq \sI^p$ for all $p\geq 1$. This inclusion is
sharp in the sense that we cannot expect $e<\codim_X Z$ in general,
which has been confirmed by the results of Bocci and Harbourne
\cite{Harbourne:CompSymPow} that no single real number less than $e$
can work for all $\sI$ of codimension $e$. The following is a typical
example discussed with Lawrence Ein, and we will give a general version and a proof in the last section.

\begin{example}\label{example01} Consider the polynomial ring $k[x_1,x_2,x_3]$ and the
ideal $I$ of the union of codimension $2$ coordinate planes, i.e.,
$$I=(x_1,x_2)\cap (x_2,x_3)\cap (x_1,x_3).$$
Then for any integer $t\geq 1$, we have $I^{(4t)}\nsubseteq
I^{3t+1}$ but $I^{(4t)}\subseteq I^{3t}$.

This shows that we cannot find a constant $c$ to obtain the
following inclusion:
$$I^{(p+c)}\subseteq I^p \quad \mbox{ for all $p$ large enough}.$$
Because otherwise, suppose we have the above inclusion. Take
$p=3t+1$, then  for all $t\geq 1$ we would have
$$I^{(3t+1+c)}\subseteq I^{3t+1},$$
and therefore we have
$$I^{(4t)}\subseteq I^{(3t+1+c)}\subseteq I^{3t+1}.$$
This is a contradiction to $I^{(4t)}\nsubseteq I^{3t+1}$.

This example can also be deduced from the work of Bocci and Harbourne
\cite{Harbourne:CompSymPow}, \cite{Harbourne:ResIdelPts}, where they
consider the homogeneous ideal $I$ of points in projective space cut
out by generic hyperplanes and gave a criterion when $I^{(r)}\subset
I^m$.
\end{example}

If symbolic powers are almost the same as ordinary powers, then we can easily obtain regularity bounds for symbolic powers. The statement of the following theorem was suggested by the referee, which is more clear than its original form in the first draft of the paper.

\begin{theorem}\label{thm45} Let $\sI$ be an ideal sheaf on $\nP^n$ and let
$s=s(\sI)$ be the $s$-invariant. Suppose that except at an isolated set of points the symbolic power $\sI^{(p)}$ agree with the ordinary power $\sI^p$ for $p$ large enough. Then there exists a constant $e$ such that for all $p\geq 1$, one has $\reg \sI^{(p)}\leq sp +e$.
\end{theorem}
\begin{proof}Consider a short exact sequence
\begin{center}
$0\rightarrow \sI^p\rightarrow \sI^{(p)}\rightarrow Q\rightarrow 0.$
\end{center}
By assumption we see that the quotient $Q$ has $\dim \Supp Q\leq 0$. Thus $Q$ has no higher cohomology groups. Then we have $\reg \sI^{(p)}\leq \reg \sI^p$, and the result follows from Theorem \ref{b2}.
\end{proof}

In order to see when an ideal sheaf satisfies the condition in Theorem \ref{thm45}, we consider an algebraic set of $X$,
$$\Nlci(\sI)=\{x\in X|\sI \mbox{\ \ is not a local complete intersection at $x$} \}.$$
We use the convention that if $\sI$ is trivial at $x$ then $\sI$ is a
local complete intersection at $x$. This algebraic set will be used
to control the set where ordinary powers are not equal to symbolic
powers. The main criterion for comparing ordinary and symbolic
powers is established in the work of Li and Swanson \cite{Swanson:SymbPowRadIdeal}, which generalizes the early work of Hochster \cite{Hochster:CrEqOrdSymPwer}. We cite this criterion here in the form used later.

\begin{lemma}[{\cite[Corollary 3.8]{Swanson:SymbPowRadIdeal}}]\label{thm02} Assume that an ideal sheaf $\sI$ on a nonsingular variety $X$
defines a reduced subscheme. For any point $x\in X$ such
that $x$ is not in $\Nlci(\sI)$, we have
\begin{center}
$\sI^p_x=\sI^{(p)}_x, \quad\mbox{for all }p\geq 1.$
\end{center}
\end{lemma}

From this lemma, we see that the set $\Nlci(\sI)$ covers the points
where $\sI^p$ is not equal to $\sI^{(p)}$ for some $p\geq 1$. Now we can easily get the following corollaries of Theorem \ref{thm45}.

\begin{corollary}Let $\sI$ be an ideal sheaf on $\nP^n$ and let
$s=s(\sI)$ be the $s$-invariant. Assume that $\sI$ defines a reduced subscheme and $\dim \Nlci(\sI)\leq 0$. Then there exists a constant $e$ such that for all $p\geq 1$, one has $\reg \sI^{(p)}\leq sp +e$.
\end{corollary}

And easily we obtain Theorem 2 in Introduction as follows.
\begin{corollary}Let $\sI$ be an ideal sheaf on $\nP^n$ and let
$s=s(\sI)$ be the $s$-invariant. Assume that $\sI$ defines a reduced subscheme of dimension $\leq 1$. Then there exists a constant $e$ such that for all $p\geq 1$, one has $\reg \sI^{(p)}\leq sp +e$.
\end{corollary}

\begin{proof} Let $Z$ be the subscheme defined by $\sI$. Then the irreducible components of $Z$ are distinct points or reduced irreducible curves. Thus from Lemma \ref{thm02} except for those finitely many points which are singular points of each dimension $1$ component and intersection points of two dimension $1$ components, $\sI^p$ is equal to $\sI^{(p)}$ for all $p\geq 1$. Then the result follows from Theorem \ref{thm45}.
\end{proof}

\begin{remark} Typical low dimensional varieties satisfying the
hypothesis of Theorem \ref{thm45} are integral curves, normal surfaces and terminal threefolds. It would be very interesting to know if the
bound in Theorem \ref{b2} works for symbolic powers of any ideal sheaf. We need some
new ideas to solve this problem. However, we propose a conjecture in
this direction.
\end{remark}

\begin{conjecture} Let $\sI$ be an ideal sheaf defining a reduced
subscheme $Z$ of $\nP^n$ and let $s=s(\sI)$ be the $s$-invariant.
Then there is a constant $e$ such that for all $p\geq 1$, one has
$$\reg \sI^{(p)}\leq  sp +e.$$
\end{conjecture}

\section{A vanishing theorem of ideal sheaves}

\noindent In the present section, we give a vanishing theorem of
powers of an ideal sheaf. It is inspired by the work of Ein and deFernex
\cite{Ein:VanishLCPairs} on a vanishing theorem for log canonical
pairs, and generalizes the result of \cite{BEL} in the case of
nonsingular varieties. We start by recalling some preliminary
definitions and basic properties, where we follow notation from \cite{Ein:VanishLCPairs} for the convenience of the
reader.

Consider a pair $(X,Z)$, where $X$ is a normal, $\nQ$-Gorenstein
variety and $Z$ is a formal finite sum $Z=\sum_j q_j Z_j$ of proper
closed subschemes $Z_j$ of $X$ with nonnegative rational
coefficients $q_j$. Take a log resolution $f:X'\rightarrow X$ of the
pair $(X,Z)$ and denote by $K_{X'/X}$ the relative canonical divisor
of $f$, such that each scheme-theoretical inverse image
$f^{-1}(Z_j)$ and the exceptional locus of $f$ are divisors supported on a single simple normal crossings divisor. For a prime divisor $E$ on $X'$, we
denote its coefficient in $f^{-1}(Z_j)$ by $\Val_E(Z_j)$ or
$\Val_E(\sI_{Z_j})$, where $\sI_{Z_j}$ is the ideal sheaf of $Z_j$
in $X$. We also denote $E$'s coefficient in $K_{X'/X}$ by
$\ord_E(K_{X'/X})$. We call the set $f(E)$ the center of $E$ in $X$
and write it as $C_X(E)$. The pair $(X,Z)$ is said to be {\em log
canonical} if for any prime divisor $E$ on $X'$, the coefficient of
$E$ in $K_{X'/X}- \sum q_jf^{-1}(Z_j)$ is $\geq -1$. In particular,
$X$ is said to be log canonical if the pair $(X,0)$ is log
canonical.

If $X$ is nonsingular, then the {\em multiplier
ideal sheaf} $\mJ(X,Z)$ associated to the pair $(X,Z)$ is defined by
$$\mJ(X,Z)=f_*\sO_{X'}(K_{X'/X}-\lfloor \sum q_jf^{-1}(Z_j)\rfloor)\subseteq\sO_X.$$
The following vanishing theorem is a multiplier ideal sheaf version
of Kawamata-Viehweg vanishing theorem.

\begin{Nadel} Assume the pair $(X,Z)$ as above and suppose that $X$
is a nonsingular projective variety. Let $L_j$ and $A$ are Cartier
divisors on $X$ such that $\sI_{Z_j}\otimes L_j$ is globally
generated for each $j$ and $A-\sum q_jL_j$ is big and nef. Then
$$H^i(X,\omega_X\otimes \sO_X(A)\otimes \mJ(X,Z))=0\quad\quad\mbox{ for \ }i>0.$$
\end{Nadel}

Now, following the idea in \cite{Ein:VanishLCPairs}, we are able to give our main theorem in this section. We start with the following easy lemma which has a quicker proof, suggested by the referee, than one in the first draft of the paper.

\begin{lemma}\label{thm06} Let $X$ be a nonsingular projective
variety and $V\subset X$ be a normal local complete intersection subvariety of codimension $e$. Suppose that $V$ is scheme-theoretically given by $V=H_1\cap\cdots\cap H_t,$ for some $H_i\in|L^{\otimes d_i}|$, where $L$ is a globally
generated line bundle on $X$ and $d_1\geq\cdots\geq d_t$. Then $V$ is log canonical if and only if the pair $(X,eV)$ is log canonical.
\end{lemma}
\begin{proof} Using \cite[Corollary 3.2]{EM:InvAdjLCI} we see that for any point $p\in V$, we have $\mld(p;X,eV)=\mld(p;V,0)$. Immediately, $V$ is log canonical if and only if the pair $(X,eV)$ is log canonical.
\end{proof}

\begin{theorem}\label{thm1.1} Let $X$ be a nonsingular projective
variety and $V\subset X$ be a local complete intersection subvariety
with log canonical singularities. Suppose that $V$ is
scheme-theoretically given by
$$V=H_1\cap\cdots\cap H_t,$$
for some $H_i\in|L^{\otimes d_i}|$, where $L$ is a globally
generated line bundle on $X$ and $d_1\geq\cdots\geq d_t$. Set
$e=\codim_X V$, then we have
$$H^i(X,\omega_X\otimes L^{\otimes k}\otimes A\otimes \sI^{p}_V)=0,\mbox{ for $i>0,\ k\geq pd_1+d_2+\cdots+d_e$},$$
where $p\geq1$ and $A$ is a nef and big line bundle on $X$.
\end{theorem}

\begin{proof} First of all, note that by the assumption and Lemma \ref{thm06}, we have $V$ is log canonical if and only if the pair $(X,eV)$ is log canonical.

Consider the base locus subscheme $B\subset X$ of the linear series
$|L^{\otimes (d_1+\cdots+d_e)}\otimes\sI^e_V|$. For each $p\in V$
using \cite[Corollary 3.5 or Proposition 3.1]{Ein:VanishLCPairs}, we
see that there is a divisor $D\in
|L^{\otimes(d_1+\cdots+d_e)}\otimes\sI^e_V|$ such that the pair $(X,D)$ is
log canonical at $p$. This implies that the pair $(X,B)$ is also log
canonical at $p$ and therefore  is log canonical in a neighborhood
of $V$.

Take a log resolution $\mu: X'\rightarrow X$ of $B$ and $V$ such
that the scheme-theoretical inverse images $\mu^{-1}(B)$ and
$\mu^{-1}(V)$ and the exceptional locus of $\mu$ are divisors supported on a single simple normal crossings divisor. Then $\mu$ factors through the
blowing-up of $X$ along $V$. There is a unique weil divisor on the blowing-up dominating $V$. Let $F$ be the strict transformation of this divisor on $X'$. We have the following two observations.
\begin{enumerate}
\item [(i)] For any divisor $E$ on $X'$, we have
$\Val_E B\geq e\Val_EV$, since $\sI_B\subseteq\sI^e_V$ by the
definition of $B$.
\item [(ii)] In particular, for the divisor $F$, we have
$\Val_F B= e\Val_FV=e$.
\end{enumerate}

Now, we construct for $0<\delta\ll 1$, a formal sum
$$Z=(1-\delta)B+\delta e V+(p-1)V, \mbox{ \ for } p\geq1,$$
and associate to $Z$ the multiplier ideal sheaf $\mJ(X,Z)$. We
compare $\mJ(X,Z)$ with $\sI^p_V$ locally around $V$. For this, let
$U$ be a neighborhood of $V$ such that the prime divisors in
$$K_{X'/X}-(1-\delta)\mu^{-1}(B)+\delta e \mu^{-1}(V)+(p-1)\mu^{-1}(V)$$
over $U$ have centers intersecting $V$ and the pair $(X,B)$ is log
canonical in $U$. Picking a such prime divisor $E$ on $X'$, there
are two possibilities for its center.
\begin{enumerate}
\item Assume that $C_X(E)\subseteq V$. Then  $\val_EV\geq 1$. Since the pair $(X,B)$ is log
canonical around $V$, we have $\Val_EB-\ord_E K_{X'/X}\leq1$, and
therefore $\Val_EB-\ord_E K_{X'/X}\leq\Val_EV$. Thus
\begin{eqnarray*}
&&\Val_E((1-\delta)B+\delta eV+(p-1)V)-\ord_E K_{X'/X}\nonumber\\
&\leq &\val_EB-\ord_EK_{X'/X}+\val_E(p-1)V \nonumber\\
&\leq &  \Val_E pV. \nonumber
\end{eqnarray*}
Then we have $\ord_EK_{X'/X}-\val_EZ\geq-\val_E\sI^p_V$.
\item Assume that $C_X(E)\cap V$ is not empty but $C_X(E)\nsubseteq
V$. Then $\Val_EV=0$. We see that
\begin{eqnarray*}
&&\Val_E((1-\delta)B+\delta eV+(p-1)V)-\ord_EK_{X'/X}\nonumber\\
&= &\Val_E((1-\delta)B)-\ord_EK_{X'/X} < 1.\nonumber
\end{eqnarray*}
The last inequality is because the pair $(X,B)$ is log canonical in
$U$ and therefore the pair $(X,(1-\delta)B)$ is Kawamata log
terminal in $U$. Hence we obtain $\ord_EK_{X'/X}-\val_EZ>-1$.
\end{enumerate}
Combining possibilities (1) and (2) above, we obtain that for any
divisor $E$ over $U$
$$\ord_EK_{X'/X}-\lfloor\Val_EZ\rfloor\geq-\Val_E\sI^p_V.$$
This implies that on $U$, we have the inclusion
$$\sI^p_V|_U\subseteq\mJ(X,Z)|_U.$$

Next we prove globally on $X$, $\mJ(X,Z)\subseteq \sI^{p}_V$. From
the definition of multiplier ideal sheaves and the fact that
$\sI_B\subseteq \sI^e_V$, we have
$$\mJ(X,Z)\subseteq \mJ(X,eV+(p-1)V).$$

Let $\eta$ be the generic point of $V$. Take a neighborhood $U'$ of
$\eta$ in $X$ such that $V|_{U'}$ is nonsingular. The blowing-up of
$U'$ along $V|_{U'}$ gives a log resolution of the pair ($U', V|_{U'})$.
Computing $\mJ(X,eV+(p-1)V)$ on this blowing-up, we see that at the
point $\eta$,
$$\mJ(X,eV+(p-1)V)_{\eta}=\sI^p_{V,\eta}.$$
Thus globally on $X$, $\mJ(X,eV+(p-1)V)\subseteq \sI^{(p)}_V$. Since
$V$ is a local complete intersection, $\sI^p_V=\sI^{(p)}_V$ and
therefore we conclude that on $X$
$$\mJ(X,Z)\subseteq \sI^p_V.$$

From arguments above, in the open neighborhood $U$ of $V$, we have
the equality $\mJ(X,Z)|_U=\sI^p_V|_U$ and therefore
$\mJ(X,Z)=\sI^p_V \cap \sI_W$ for some subscheme $W$ of $X$ disjoint
from $V$.

Applying Nadel's Vanishing Theorem to $\mJ(X,Z)=\sI^p_V \cap \sI_W$
and using \cite[Lemma 4.3]{Ein:VanishLCPairs}, we have the vanishing
$$H^i(X,\omega_X\otimes L^{\otimes k}\otimes A\otimes \sI^{p}_V)=0,\mbox{ for $i>0,\ k\geq pd_1+d_2+\cdots+d_e$},$$
where $p\geq1$ and $A$ is a nef and big line bundle on $X$.
\end{proof}

\begin{remark} Taking $X$ to be the projective space $\nP^n$ and
assuming that $V$ is a local complete intersection with log
canonical singularities cut out by hypersurfaces of degrees $d_1\geq
d_2\geq\cdots\geq d_t$ of codimension $e$, we get Corollary
\ref{intro:cor} in the Introduction. In particular, if $V$ is
nonsingular, we recover the result in \cite{BEL}.
\end{remark}

\section{An example of ordinary and symbolic powers}

\noindent In this last section, we construct an example discussed
with Lawrence Ein, which enables us to compare ordinary and symbolic
powers precisely. It generalizes Example \ref{example01} and offers ideals of different codimension. This section is kind of appendix but we hope this example will be useful in the study of symbolic powers.

\begin{example}\label{thm40} Let $k[x_1,x_2,\cdots,x_n]$ be a polynomial ring.
Let $1\leq e\leq n-1$ be an integer. Define a set of $e$
multi-indices
$$\Sigma=\{ (i_1,\cdots,i_e)\ |\ 1\leq i_1<i_2<\cdots<i_e\leq n\}.$$
For any $\sigma \in \Sigma$, set
$I_{\sigma}=(x_{i_1},x_{i_2},\cdots,x_{i_e})$. Consider the ideal of
the union of codimension $e$ coordinate planes, i.e.,
$$I=\bigcap_{\sigma\in \Sigma}I_{\sigma}=(x_{j_1}x_{j_2}\cdots x_{j_{n-e+1}}\ |\ 1\leq j_1<j_2<\cdots<j_{n-e+1}\leq n).$$
Set $d=n-e+1$. Then for all $t\geq 1$, we have
\begin{enumerate}
\item [(i)] $I^{(edt)}\nsubseteq I^{nt+1}$, and
\item [(ii)]$I^{(edt)}\subseteq I^{nt}$.
\end{enumerate}
\begin{proof}
(i). Set a monomial $y=x_1x_2\cdots x_n$, then $y^{dt}\in
I^{(edt)}$. Note that $\deg y^{dt}=ndt$, but $I^{nt+1}$ is generated by monomials of degree $d(nt+1)=ndt+d$. Thus $y\notin I^{nt+1}$
and therefore $I^{(edt)}\nsubseteq I^{nt+1}$.

(ii) We first need to prove the following lemma to describe the generators of
$I^{nt}$.
\begin{lemma} Let $x=x^{b_1}_1 x^{b_2}_2\cdots x^{b_n}_n$ be a
monomial. Then $x$ is a minimal generator of $I^{nt}$ if and only if
\begin{enumerate}
\item $0\leq b_i\leq nt$ for $i=1,\ldots,n$,
\item $b_1+b_2+\cdots+b_n=ndt$.
\end{enumerate}
That is 
$$I^{nt}=(x^{b_1}_1 x^{b_2}_2\cdots x^{b_n}_n\ |\  b_1+b_2+\cdots+b_n=ndt,\ 0\leq b_i\leq nt,\ i=1,\ldots,n).$$
\end{lemma}
\begin{proof}[Proof of Lemma] Since ``only if'' is easy, we prove ``if"
part.

We denote by a vector $\alpha=(\alpha_1,\alpha_2,\cdots,\alpha_n)$
the powers in a monomial
$x^{\alpha}=x^{\alpha_1}_1x^{\alpha_2}_2\cdots x^{\alpha_n}_n$. Set
$\beta_0=(b_1,b_2,\cdots,b_n)$. We need to show that $x=x^{\beta_0}$
is a generator of $I^{nt}$. It suffices to show $x^{\beta_0}\in
I^{nt}$ because of condition $(2)$ and because $I$ has the form
$$I=(x_{j_1}x_{j_2}\cdots
x_{j_d}\ |\ 1\leq j_1<j_2<\cdots<j_d\leq n).$$

To this end, we need to show we can remove minimal generators of $I$ from
$x$ by $nt$ times and then we get $0$. In another word,
we need to remove the vectors in the set
$$U=\{(u_1,u_2,\cdots,u_n)\ |\ u_j=0 \mbox{ or }1,\  1\leq j\leq n, \mbox{ and } u_1+u_2+\cdots+u_n=d\}$$
from $\beta_0$ by $nt$ steps to end at
$(0,0,\cdots,0)$. For $i\geq 0$, let
$\beta_i=(b^i_1,b^i_2,\cdots,b^i_n)$ be the resulting vector of the $i$-th step of such removing. Now we describe
how to remove a vector in $U$ for each step. We start with $\beta_0$ and suppose
we have obtained $\beta_i\neq 0$, let $m_1,\cdots,m_d$ be the index
such that $b^i_{m_1},\cdots,b^i_{m_d}$ are the maximal $d$ numbers
in the vector $\beta_i$. Let $u^i=(u_1,u_2,\cdots,u_n)\in U$ such
that $u_{m_1}=1,\cdots,u_{m_d}=1$. Then we remove $u^i$ from
$\beta_i$ to define $\beta_{i+1}=\beta_i-u^i$. This method works
because we observe inductively that
\begin{enumerate}
\item $b^i_1+b^i_2+\cdots+b^i_n=ndt-di$, and
\item $0\leq b^i_j\leq nt-i$.
\end{enumerate}
The condition (2) guarantees that these maximal $d$ numbers
$b^i_{m_1},\cdots,b^i_{m_d}$ of $\beta_i$ are always nonzero unless
$\beta_i=0$.

Thus following the above method, we finally achieve $(0,\cdots,0)$
after $nt$ steps and therefore $x=x^{\beta_0}\in I^{nt}$. This
finishes the proof the lemma.
\end{proof}

Having the above lemma in hand, now we can show $I^{(edt)}\subseteq I^{nt}$.
Pick a monomial
$$x=x^{a_1}_1x^{a_2}_2\cdots x^{a_n}_n\in I^{(edt)}.$$
Since $I^{(edt)}=\cap_{\sigma\in \Sigma}I^{edt}_{\sigma}$, $x$ sits
in $I^{edt}_{\sigma}=(x_{i_1},x_{i_2},\cdots,x_{i_e})^{edt}$ for
each $\sigma\in \Sigma$. This implies
$$a_{i_1}+a_{i_2}+\cdots+a_{i_d}\geq edt,\quad \mbox{ for any }\sigma=(i_1,\cdots,i_e)\in \Sigma.$$
Adding those inequalities together, we obtain
$$a_1+\cdots+a_n\geq \frac{edt \cdot |\Sigma|}{{n-1\choose e-1}}=\frac{edt\cdot {n\choose e}}{{n-1\choose e-1}}=ndt.$$
We assume without loss of generality that $a_1\geq a_2\geq \cdots
\geq a_n$. Then we define
\begin{eqnarray*}
b_n & = & \min\{ a_n,nt\}\\
b_{n-1} & = & \min\{ a_{n-1},ndt-b_n,nt\}\\
   & \cdots &  \\
b_2 & = & \min\{ a_2,ndt-b_n-b_{n-1}-\cdots-b_3,nt\}\\
b_1 & = & \min\{ a_1,ndt-b_n-b_{n-1}-\cdots-b_2,nt\}
\end{eqnarray*}
Set $b=(b_1,b_2,\cdots,b_n)$ and $x^b=x^{b_1}_1x^{b_2}_2\cdots
x^{b_n}_n$. Then $x^b\ |\ x$. Also note that
\begin{enumerate}
\item $b_1+b_2+\cdots+b_n=ndt$, and
\item $0\leq b_i\leq nt$ for $1\leq i\leq n$.
\end{enumerate}
Thus by Lemma above, $x^b\in I^{nt}$ and therefore $x\in
I^{nt}$. This gives us the inclusion $I^{(edt)}\subseteq I^{nt}$.
\end{proof}
\end{example}

\begin{remark}(1) Applying a slight modification of the above proof,
we are able to show that if $r\cdot \frac{n}{e}\geq dm$, then
$I^{(r)}\subset I^m$.

(2) There is a conjecture due to Harbourne \cite[Conjecture
8.4.3]{Harbourne:PrimSeshadriConst} that for any homogeneous ideal
$I$ of codimension $e$, one has $I^{(r)}\subset I^m$ if $r\geq
em-(e-1)$. It is true if $I$ is a monomial ideal.

(3) The part (i) of Example \ref{thm40} has been
considered and generalized in the work of Bocci and Harbourne
\cite[Theorem 2.4.3(b)]{Harbourne:CompSymPow}, replacing the $n$
coordinate hyperplanes here by any number $s\geq n$ of generic
hyperplanes.
\end{remark}

\end{document}